\documentclass[a4paper,reqno,oneside]{amsart}
\usepackage[
    style=numeric,
    sorting=nty,
    maxnames=99,
    maxalphanames=5,
    sortcites,
    url=false]{biblatex}
\usepackage[british]{babel}
\usepackage{csquotes}
\numberwithin{equation}{section}
\usepackage{amssymb}
\usepackage[protrusion=true,expansion=true]{microtype}
\usepackage{amsmath,amsfonts,amsthm}
\usepackage[pdftex]{graphicx}
\usepackage{url}
\usepackage{svg}
\usepackage[title,titletoc,page]{appendix}
\usepackage{listings}
\usepackage{color}
\usepackage{caption}
\usepackage{subcaption}
\usepackage{dsfont} 
\usepackage{bm}
\usepackage{upgreek}
\usepackage[normalem]{ulem}
\usepackage{enumerate}
\usepackage{mathtools}
\usepackage{lipsum}
\usepackage{adjustbox}
\usepackage{tikz}

\usepackage[textsize=scriptsize]{todonotes}

\usepackage{circuitikz}
\usepackage{tcolorbox}
\usetikzlibrary{tikzmark,calc}
\usepackage{xargs}      
\usepackage{algorithm}
\usepackage{algpseudocode}
\usepackage{nicematrix}
\usepackage[T1]{fontenc} 
\usepackage{lmodern} 
\rmfamily
\DeclareFontShape{T1}{lmr}{b}{sc}{<->ssub*cmr/bx/sc}{}
\DeclareFontShape{T1}{lmr}{bx}{sc}{<->ssub*cmr/bx/sc}{}
\usepackage[left=2.4cm,right=2.4cm,top=2.5cm,bottom=2.5cm,footskip=1.0cm,headsep=1cm,marginparwidth=1.9cm]{geometry}
\usepackage[%
pdfauthor={Yannick De Bruijn, Erik Orvehed Hiltunen},%
pdftitle={Numerical Toeplitz Symmetriser},
pdfsubject={Numerical Toeplitz Symmetriser},
pdfkeywords={banded Toeplitz matrices, generalised Brillouin zone, non-Bloch band theory, non-Hermitian
skin effect, eigenvalue asymptotics, spectral distance}]{hyperref}
\usepackage{orcidlink}
\usepackage{fancyhdr}
\fancypagestyle{plain}{%
\fancyhead[C]{}
\fancyfoot[C]{}

}
\fancypagestyle{nosection}{%
\fancyhead[CE]{}
\fancyhead[CO]{}
\fancyfoot[CE,CO]{\thepage}
}
\numberwithin{equation}{section}		
\numberwithin{figure}{section}			
\numberwithin{table}{section}			

\newcommand{\vect}[1]{\boldsymbol{\mathbf{#1}}}

\newcommand{\Cb}{\mathbb{C}}
\DeclareMathOperator{\C}{\mathbb{C}}

\newcommand{\N}{\mathbb{N}}
\renewcommand{\O}{\mathcal{O}}

\newcommand{\R}{\mathbb{R}}

\newcommand{\Z}{\mathbb{Z}}

\renewcommand{\epsilon}{\varepsilon}
\newcommand{\dx}{\: \mathrm{d}}

\renewcommand{\hat}{\widehat}
\renewcommand{\i}{\mathrm{i}}
\renewcommand{\d}{\mathrm{d}}
\renewcommand{\Im}{\mathfrak{Im}}

\newcommand{\svdots}{\raisebox{3pt}{\scalebox{.6}{$\vdots$}}}
\newcommand{\sddots}{\raisebox{3pt}{\scalebox{.6}{$\ddots$}}}

\newcommand{\sdots}{\raisebox{3pt}{\scalebox{.6}{$\dots$}}}

\usepackage[noabbrev,capitalize]{cleveref}
\crefname{proposition}{Proposition}{Propositions}
\crefname{equation}{}{}
\newtheorem{theorem}{Theorem}[section]
\newtheorem{lemma}[theorem]{Lemma}

\theoremstyle{definition}
\newtheorem{definition}[theorem]{Definition}

\newtheorem{assumption}[theorem]{Assumption}
\newtheorem{remark}[theorem]{Remark}
\crefname{assumption}{Assumption}{Assumptions}
\crefname{definition}{Definition}{Definitions}
\crefname{corollary}{Corollary}{Corollaries}
\crefname{enumi}{item}{items}
\begin{document}
\title[Non-Bloch band theory for finite banded non-Hermitian Toeplitz matrices]{Non-Bloch band theory for finite banded non-Hermitian Toeplitz matrices}

\author[Y. de Bruijn]{Yannick de Bruijn$^*$ \orcidlink{0009-0009-6194-0761}}
\address{\parbox{\linewidth}{Yannick de Bruijn\\
Department of Mathematics, University of Oslo, Moltke Moes vei 35, 0851 Oslo, Norway.
\\
\href{https://orcid.org/0009-0009-6194-0761}{orcid.org/0009-0009-6194-0761}}}
\email{yannicd@math.uio.no}

\author[E. O. Hiltunen]{Erik Orvehed Hiltunen \orcidlink{0000-0003-2891-9396}}
\address{\parbox{\linewidth}{Erik Orvehed Hiltunen\\
Department of Mathematics, University of Oslo, Moltke Moes vei 35, 0851 Oslo, Norway. \\
\href{http://orcid.org/0000-0003-2891-9396}{orcid.org/0000-0003-2891-9396}.}}
\email{erikhilt@math.uio.no}

\begin{abstract}
Non-Bloch band theory recovers the open-boundary spectrum of a non-Hermitian chain from its Bloch Hamiltonian evaluated along the generalised Brillouin zone rather than the unit circle. We develop its finite-size counterpart for banded Toeplitz matrices with real open-boundary spectrum. Reality of the spectrum forces the generalised Brillouin zone to be parametrised by an analytic polar curve, so the known sufficient condition for a real spectrum is also necessary under some mild non-degeneracy conditions. At finite size $n$, the eigenvalues of the original and the symmetrised Bloch Hamiltonians are given by samples at $n$ quantised momenta from the same limiting band, although separated by a phase mismatch of order $1/n$ which we obtain in closed form. Iterating this correction $K$ times yields a quantisation rule on the generalised Brillouin zone reproducing the finite spectrum with an error $\mathcal{O}(n^{-K-1})$ for any fixed $K$. At finite size, the generalised Brillouin zone itself is unchanged and only its parametrisation depends on $n$.
\end{abstract}
\maketitle
\vspace{3mm}
\noindent
\textbf{Mathematics Subject Classification (MSC2020):} 47B35,   
15B05,   
15A18,   
41A25,   
65F15,   
81Q12.   

\vspace{3mm}
\noindent
\textbf{Keywords:} banded non-Hermitian Toeplitz matrices, generalised Brillouin zone, non-Bloch band theory, non-Hermitian skin effect, eigenvalue asymptotics, spectral distance.
\vspace{5mm}

\section{Introduction}

This paper is devoted to the study of the asymptotic spectra of $m$-banded Toeplitz matrices and constitutes a continuation of the work presented in \cite{GBZfoundation}.
Non-Bloch band theory has attracted considerable interest because it makes non-reciprocal systems easier to handle \cite{NonBlochBandTheory}, and particular attention has been paid to relating open to periodic boundary conditions. These descriptions, however, concern genuinely infinite structures, whereas every experimental system is necessarily finite. It is therefore not clear what the theory predicts for a finite chain, nor how it converges to its infinite counterpart.

The generalised Brillouin zone (GBZ) has demonstrated itself a useful tool to study the localisation in periodic media \cite{DEBRUIJN2026110014, debruijn2024complexbandstructuresubwavelength}, and in the case of tridiagonal $k$-Toeplitz matrices \cite{ammari2024spectra, debruijn2025complexbandstructurelocalisation}.
It has long been known that a class of banded non-Hermitian Toeplitz matrices possesses real spectrum, and sufficient conditions for reality were given in \cite{GBZfoundation}. There it was also shown that the open limit can be related to periodic boundary conditions by evaluating the symbol along a curve in the complex plane, the generalised Brillouin zone, instead of along the unit circle. Moreover, it was shown that the generalised Brillouin zone preserves the density of states, that is the asymptotic eigenvalue distribution. On a Toeplitz matrix level, the generalised Brillouin zone produces an asymptotic similarity transform of the original, non-Hermitian matrix to a symmetric Toeplitz matrix. However, this is a statement about the limit alone, and not much has been said about eigenvalues of corresponding finite systems.

This manuscript closes that gap by quantifying how the generalised Brillouin zone acts on the finite matrix level and how it preserves the eigenvalues pointwise. Indeed, the eigenvalues of the original problem and the transformed one, though their generalised Brillouin zones agree, are sampled differently from these bands. Using the recent results of \cite{BogoyaGrudsky2025, BOGOYA_SLM}, we correct the parametrisation of the generalised Brillouin zone such that the sampling before and after the transform agree.
Iterating the correction improves the pointwise eigenvalue error to any polynomial order. For every fixed order $K$ and all large enough $n$  there exists a reparametrisation of the generalised Brillouin zone reproducing the finite non-reciprocal spectrum with pointwise error $\mathcal{O}(n^{-K-1})$.
The generalised Brillouin zone is thus rigid at finite size and only the momenta at which it is sampled are size-dependent.

This manuscript is organised as follows. We start in Section \ref{Sec: Toeplitz Theory} by recalling the known results on Toeplitz theory and by stating the standing assumption of the manuscript. Moreover, we provide a proof that, under the assumptions given, the reality criterion from \cite{GBZfoundation} is indeed optimal. In Section \ref{sec:sharp} we will recall the pointwise eigenvalue characterisation given by \cite{BogoyaGrudsky2025}, and use these results to reparametrise the generalised Brillouin zone allowing for pointwise eigenvalue convergence. We conclude the section by presenting an iteration which achieves tangential corrections of the generalised Brillouin zone so that the eigenvalues converge pointwise at any arbitrary polynomial order. In Section \ref{Sec: Physical Applications} we discuss the implications for condensed-matter physics and motivate the finite size non-Bloch band theory.
The Matlab code supporting the findings of this article is available in a repository linked in Section \ref{Sec: Data availability}.

\section{Spectra of banded Toeplitz matrices}\label{Sec: Toeplitz Theory}
An $m$-banded Toeplitz operator is generated by the finite Laurent polynomial,
\begin{equation}\label{eq: def symbol function banded}
    f_m(z) = \sum_{k = - m}^m a_k z^{-k}, \quad a_{\pm m} \neq 0,
\end{equation}
for a sequence $\{a_k\}_{k\in \Z}\subset \R$ such that the $m$-banded Toeplitz operator is given by the semi-infinite matrix,
\begin{equation} \label{eq: def banded Toeplitz Operator}
    \mathbf{T}(f_m) = \left(\begin{smallmatrix} 
        a_0     & a_{-1}  & \sdots & a_{-m}  &   \\
        a_1     & \sddots & \sddots &       & \sddots \\
        \svdots & \sddots & \sddots & \sddots &  \\
        a_m     &         &         &         &  \\
                & \sddots &         &         &         
    \end{smallmatrix}\right).
\end{equation}
To define a finite Toeplitz matrix, we introduce the projection operator,
\begin{align}
    \vect{P}_n: \ell^2(\N, \mathbb{C}) &\to \ell^2(\N, \mathbb{C})\\
    (x_1, x_2, x_3, \dots) &\mapsto (x_1, \dots, x_n, 0, 0, \dots).
\end{align}
A finite Toeplitz matrix is derived from a Toeplitz operator by performing the following truncation,
\begin{equation}\label{eq: truncate operator to matrix}
    \vect{T}_{n}(f) := \vect{P}_{n} \vect{T}(f) \vect{P}_{n}.
\end{equation}
Mathematically, the open limit is characterised as the smallest set on which the eigenvalues will cluster as the matrix dimension becomes large.
The limiting spectrum of a finite but growing Toeplitz matrix will be referred to as the \emph{open limit} and is defined as
\begin{equation}\label{def: open spectrum}
    \sigma_{\mathrm{open}}(f_m) := \lim_{n \to\infty}\sigma\bigl(\mathbf{T}_n(f_m)\bigr),
\end{equation}
where the limit is taken in the Hausdorff sense.
A characterisation of the open limit that does not rely on computing any eigenvalues, but solely on the symbol function, was given in \cite{Schmidt_Spitzer_1960}. For the convenience of the reader we recall the result.

\begin{theorem}[Schmidt-Spitzer]\label{thm: schmidt spitzer magnitude}
    Let $f_m$ be a Laurent polynomial as defined in \eqref{eq: def symbol function banded}, then
    \begin{align}
        \lim_{n \to\infty}\sigma\bigl(\mathbf{T}_n(f_m)\bigr) &= \bigcap_{r \in (0, \infty)} \sigma\Bigl(\mathbf{T}\bigl(f_m(r\mathbb{T})\bigr)\Bigr)\\
        &= \bigl\{ \lambda \in \Cb ~:~ |z_{m}(\lambda)| = |z_{m+1}(\lambda)| \bigr\} \label{eq: modulus definition},
    \end{align}
    where $z_i(\lambda)$ are the roots of the polynomial $z^m\bigl(f_m(z)-\lambda\bigr)$ sorted in ascending magnitude, that is, $0< |z_1(\lambda)| \leq |z_2(\lambda)| \leq \dots \leq |z_{2m}(\lambda)|$.
\end{theorem}

The previous theorem motivates the following definition of the generalised Brillouin zone (GBZ) (see \cite{NonBlochBandTheory} for more details on the definition) based on the open limit
\begin{equation}\label{eq: point set roots}
 \mathbf{Z}(f_m):=\bigl\{z_m(\lambda),~z_{m+1}(\lambda)\in\Cb: |z_m(\lambda)|=|z_{m+1}(\lambda)|\ \text{for some }\lambda\in\R\bigr\} \subset \C.
\end{equation}

In the sequel, we will introduce some more structure on the geometry of the set $\mathbf{Z}(f_m)$. 

\begin{definition}\label{def: polar curve}
    We define a $C^1$ \emph{polar curve} as the set parametrised by polar coordinates such that
    \begin{equation}\label{eq: def eq of polar curve}
        p(e^{\i\theta}) = r(\theta) e^{\i \theta}, \quad \theta \in [-\pi, \pi],
    \end{equation}
    for some continuously differentiable function $r: [-\pi, \pi] \to (0, \infty)$, such that $r(-\theta) = r(\theta)$ and $r'(\pi) = r'(-\pi) =0$. For ease of notation, we will often denote the polar curve as $p(\mathbb{T})$, where $\mathbb{T}$ is the one-dimensional torus. If $r(\theta)$ is moreover real-analytic as a $2\pi$-periodic function, we call $p(\mathbb{T)}$ an \emph{analytic polar curve}.
\end{definition}

Although $\mathbf{Z}(f_m)$ in \eqref{eq: point set roots} is a point set without parametrisation, sweeping $\lambda$ across $\sigma_{\mathrm{open}}(f_m)$ provides a natural one. For $\lambda\in\R$ the central roots form a conjugate pair $\{\zeta(\lambda),\overline{\zeta(\lambda)}\}$, so for $\sigma_{\mathrm{open}}(f_m)=[A,B]$ the set $\mathbf{Z}(f_m)$ is traced out once by $\zeta(\lambda)$ as $\lambda$ increases from $A$ to $B$ and by $\overline{\zeta(\lambda)}$ on the way back. We say that $\mathbf{Z}(f_m)$ is \emph{spanned by the polar curve} $p(\mathbb{T})$ if this closed curve is $p(\mathbb{T})$, i.e.\ if $\zeta(\lambda)=p(e^{\i\alpha(\lambda)})$ for a monotone bijection $\alpha:[A,B]\to[0,\pi]$.
{\color{blue} }

For the remainder of the paper we make the following standing assumption, which can also be found in \cite{BogoyaGrudsky2025}.
\begin{assumption}\label{ass:standing}
    Let $f_m$ be as in \eqref{eq: def symbol function banded}, then we assume that for every $\lambda\in\sigma_{\mathrm{open}}(f_m)$ we have strict modulus separation, that is
    \begin{equation}\label{eq:strictsep}
     |z_{m-1}(\lambda)|<|z_m(\lambda)|=|z_{m+1}(\lambda)|<|z_{m+2}(\lambda)| ,
    \end{equation}
    and that $\sigma_{\mathrm{open}}(f_m)$ is not a single point. For $m = 1$ we use the convention that $|z_0| = 0$ and $|z_{m+2}| = \infty$.
\end{assumption}
From the point of view of the Floquet-Bloch band structure, Assumption~\ref{ass:standing} means that the band functions cannot have local maxima or minima in the interior of the GBZ. Such max/min are typically associated to a strong next-to-nearest coupling \cite{kazemi2023drawing}. In fact, as will be shown in Lemma~\ref{cor:band function}, Assumption~\ref{ass:standing} implies that the band functions are strictly monotone, and we restrict attention to such systems for the remainder of the work.

Under Assumption~\ref{ass:standing}, the central pair of roots is uniformly separated in modulus from the inner and outer groups of roots for all admissible $\lambda$. We quantify this by the \emph{confluent distance}, 
\begin{equation}\label{eq: confluent distance}
    \gamma := \min_{\lambda \in \sigma_{\mathrm{open}}(f_m)}
    \min\left\{
        \log\frac{|z_{m+2}(\lambda)|}{|z_{m+1}(\lambda)|}, 
        \log\frac{|z_{m}(\lambda)|}{|z_{m-1}(\lambda)|}
    \right\}.
\end{equation}

Under the standing assumption of strict middle root separation, i.e. $\gamma >0$, we have the following property of the symbol function.
\begin{lemma}\cite[Lemma 2.6.]{GBZfoundation}\label{lem: separation implies non confluence}
    Suppose that Assumption \ref{ass:standing} holds. Then $f_m'(z) \neq 0$ for every $z \in \mathbf{Z}(f_m)\setminus\R$, and $f''_m(z) \neq 0$ for every $z \in \mathbf{Z}(f_m)\cap\R$ such that $f_m'(z) = 0$.
\end{lemma}

 Moreover, a sufficient condition for the reality of the open limit can be derived from the generalised Brillouin zone.

\begin{theorem}\cite[Theorem 2.11.]{GBZfoundation}\label{Thm: real spectra criterion}
    Suppose that $\mathbf{Z}(f_m)$ is parametrised by a $C^1$ polar curve $p(\mathbb{T})$ as in \eqref{eq: def eq of polar curve} and that Assumption \ref{ass:standing} holds. Then it holds that
    \begin{equation}\label{eq: spectral limit}
        \lim_{n\to\infty} \sigma\bigl(\mathbf{T}_n(f_m)\bigr) = \bigcup_{\alpha \in[-\pi, \pi]} f_m\bigl(p(e^{\i\alpha})\bigr) \subset\R,
    \end{equation}
    where the limit is supposed to be understood in the Hausdorff sense.
\end{theorem}

Theorem~\ref{Thm: real spectra criterion} might appear considerably weaker than the criterion originally conjectured in \cite{RealSpectrumToeplitz}, since it additionally requires the GBZ to be parametrised by a polar curve of class $C^1$, instead of the set $f_m^{-1}(\R)$ containing a Jordan curve winding around the origin. The latter criterion was stated as a theorem in \cite{RealSpectrumToeplitz}, but the proof was withdrawn in \cite{CorrectionRealAsymptoticSpectrum}; it has not been restored and counterexamples to it have since been announced.
However, the following theorem shows that, under Assumption~\ref{ass:standing}, these additional hypotheses are also necessary and the criterion of \cite{GBZfoundation} is optimal.

\begin{theorem}\label{thm: exact characterisation}
    Under Assumption \ref{ass:standing} it holds that
    \begin{equation}
        \sigma_{\mathrm{open}}(f_m) \subset \R \implies \text{GBZ is spanned by an analytic  polar curve.} 
    \end{equation}
\end{theorem}

\begin{proof}
    Let us start by noting that since the set $\sigma_{\mathrm{open}}(f_m)$ is connected \cite{bams/1183529098}, reality forces $\sigma_{\mathrm{open}}(f_m) = [A, B]$ for $A<B$ by Assumption~\ref{ass:standing}.
    Under Assumption~\ref{ass:standing} for every $\lambda$ in the interior of $ \sigma_{\mathrm{open}}(f_m) = [A, B]$ there exists an open neighbourhood $U$ such that the roots $z_m$ and $z_{m+1}$ can be analytically extended throughout $U$.    For every $\lambda$ in the open interval $(A,B)$, the central roots are nonreal and form a conjugate pair, so let us without loss of generality assume that $\alpha(\lambda) = \arg z_m(\lambda)\in(0,\pi)$.
    Let us define the \emph{Bloch ratio} as follows
    \begin{equation}
        \Phi(\lambda) := \frac{z_m(\lambda)}{z_{m+1}(\lambda)}=e^{2\i\alpha(\lambda)} .
    \end{equation}
    Then by Theorem~\ref{thm: schmidt spitzer magnitude}
    \begin{equation}
        \sigma_{\mathrm{open}}(f_m)\cap U=\bigl\{\lambda\in U:\ |\Phi(\lambda)|=1\bigr\}.
    \end{equation}
    We first show that on $U\cap\sigma_{\mathrm{open}}$ it holds that $\Phi'(\lambda) \neq 0$.   Indeed if $\Phi'(\lambda)$ has a root of order $k \geq 1$ for some $\lambda^* \in \sigma_{\mathrm{open}}\cap U$, then $\{|\Phi(\lambda)|=1\}$ consists of $2(k+1)\geq 4$ analytic arcs extending from $\lambda^*$ at equal angles, but at most two of them can be real. This would imply that there is a branch of $\{|\Phi(\lambda)|=1\}$ extending into $\C$, but by assumption $\sigma_{\mathrm{open}}(f_m) \cap U \subset \sigma_{\mathrm{open}}(f_m) \subset \R$.

    Next note that
    \begin{equation}
        \Phi'(\lambda) = 2 \i \alpha'(\lambda) \Phi(\lambda) \neq 0 \implies \alpha'(\lambda) \neq 0.
    \end{equation}
    So $\alpha(\lambda)$ is strictly monotone on $(A, B)$ as well as continuous on $[A, B]$, from which it follows that the GBZ is a polar curve as defined in \eqref{eq: def eq of polar curve}. Let us introduce the radius function as
    \begin{equation}
        r(\theta) := |z_{m}(\alpha^{-1}(\theta))|.
    \end{equation}
    This function is real-analytic on $(0, \pi)$, because $\alpha^{-1}(\theta)$ is real-analytic by the inverse function theorem and the root $z_m$ is analytic by Assumption~\ref{ass:standing}. Hence, the function $r$ is real-analytic as modulus of a non-vanishing real-analytic $\C$-valued function on $(0, \pi)$. 

    It remains to extend $r$ real-analytically across the endpoint angles $\theta_e=\arg z_e(\lambda_e)\in\{0,\pi\}$, where $z_e$ denotes the real double root at the edge $\lambda_e\in\{A,B\}$. By the edge normal form in the proof of \cite[Thm.~2.10]{GBZfoundation}, the central roots at $\lambda=\lambda_e+\delta$ are $\mu(\delta)\pm\eta(\delta)$ with $\mu,\eta^2$ real-analytic, $\mu(0)=z_e$, and $\eta^2=\bigl(2/f_m''(z_e)\bigr)\delta\bigl(1+\O(\delta)\bigr)$. 
    On the spectral side the pair is non-real, so $\eta=\i t$ with $t\in\R$, and since $\tfrac{\d(t^2)}{\d\delta}(0)=-2/f_m''(z_e)\neq0$ we may invert, that is $\mu$ is real-analytic in $t^2$. Letting $t$ range over $(-\epsilon,\epsilon)$, the curve $u(t):=\mu(t^2)+\i t$ is a real-analytic two-sided parametrisation of $\mathbf{Z}(f_m)$ through $z_e$, the sign of $t$ selecting the arc above or below the axis. 
    Its angle $\theta(t)=\theta_e+\arctan\bigl(t/\mu(t^2)\bigr)$ is odd about $\theta_e$ with $\theta'(0)=1/z_e\neq0$, so $t=t(\theta)$ is odd real-analytic on a two-sided interval $I_e\ni\theta_e$, and 
    \begin{equation}
        \hat r(\theta):=|u(t(\theta))|=\sqrt{\mu\big(t(\theta)^2\big)^2+t(\theta)^2},
    \end{equation} 
    is real-analytic on $I_e$ and, depending on $\theta$ only through $t(\theta)^2$, even about $\theta_e$. In particular $\hat r(\theta_e)=|z_e|$ and $\hat r'(\theta_e)=0$. Since $\lambda(\cdot)$ is strictly monotone on $(0,\pi)$ \cite[Lem.~2.8]{GBZfoundation}, each angle in $(0,\pi)$ carries exactly one central root in the upper half-plane, whence $r=\hat r$ on the interval $I_e\cap(0,\pi)$. By evenness of $\hat r$, the even extension $r(-\theta):=r(\theta)$ coincides with $\hat r$ on all of $I_e$ and is therefore real-analytic across $\theta_e$, with $r'(\theta_e)=0$. 
    Performing this at both edges and extending $2\pi$-periodically, $r$ is an even, real-analytic function on the unit circle with $r'(0)=r'(\pi)=0$, so it immediately follows that the GBZ given by the set $\mathbf{Z}(f_m)$ is spanned by an analytic polar curve.
\end{proof}

\begin{remark}
    We have the following bootstrapping argument. Namely, under Assumption~\ref{ass:standing}, whenever the GBZ is spanned by a $C^1$ polar curve, the open limit is real-valued by Theorem~\ref{Thm: real spectra criterion} and hence the GBZ is in fact an analytic polar curve by Theorem~\ref{thm: exact characterisation}.
\end{remark}

We get the following properties of the band function.

\begin{lemma}\label{cor:band function}
    Suppose Assumption~\ref{ass:standing} holds and $\sigma_{\mathrm{open}}(f_m)\subset\R$. Let $p(e^{\i\theta})=r(\theta)e^{\i\theta}$ be the analytic polar parametrisation of $\mathbf{Z}(f_m)$ from Theorem~\ref{thm: exact characterisation} and define the band function
    \begin{equation}\label{eq: band function}
        \lambda(\theta):=f_m\bigl(r(\theta)e^{\i\theta}\bigr), \qquad \theta\in\R.
    \end{equation}
    Then $\lambda$ is real-analytic on $\R$, strictly monotone on $(0,\pi)$, and $\lambda''(\theta_e)\neq0$ for $\theta_e\in\{0,\pi\}$.
\end{lemma}

\begin{proof}
    Real-analyticity of $\lambda$ is inherited from $r$. Writing $\zeta(\theta):=r(\theta)e^{\i\theta}$, we have
    \begin{equation}
        \lambda'(\theta)=f_m'\bigl(\zeta(\theta)\bigr)\zeta'(\theta), \qquad \zeta'(\theta)=\bigl(r'(\theta)+\i r(\theta)\bigr)e^{\i\theta}\neq0,
    \end{equation}
    since $r>0$, so $\lambda' \neq 0$ on $(0, \pi)$ by Lemma~\ref{lem: separation implies non confluence}.
    The monotonicity as well as $\lambda''(\theta_e)\neq0$ for $\theta_e\in\{0,\pi\}$ follow by Lemma~\ref{lem: separation implies non confluence}.
\end{proof}

By construction, the symmetrised symbol function $g := f_m\circ p$ is real-valued but no longer has finitely many non-vanishing Fourier coefficients. However, by Theorem~\ref{thm: exact characterisation} the parametrisation $p: \mathbb{T}\to \C$, is analytic and bounded away from $0$, so that $g$ is analytic on $\mathbb{T}$.
Consequently, the Fourier coefficients of $g$ decay exponentially, that is, there are constants $C, \gamma_F>0$ such that 
\begin{equation}\label{eq: fourier coefficients of g}
    b_k:=\widehat{(f_m\circ p)}_k,\qquad |b_k|\le Ce^{-\gamma_F|k|}.
\end{equation}
In particular, for every $\beta>0$,
\begin{equation}\label{eq: wiener membership}
    \|g\|_{W^\beta}:=\sum_{k\in\Z}|b_k|\bigl(|k|+1\bigr)^{\beta}<\infty ,
\end{equation}
that is, $g$ lies in the weighted Wiener algebra $W^{\beta}$ for every exponent $\beta>0$. Together with Lemma~\ref{lem: separation implies non confluence}, which gives $\lambda'\neq0$ on $(0,\pi)$ and $\lambda''(0)\neq0\neq\lambda''(\pi)$, this places $g$ in the simple-loop class $\mathrm{SL}^{\beta}$
for every $\beta>0$. The results of \cite{BOGOYA_SLM} therefore apply to $g$ directly, without any preliminary truncation to finite bandwidth of the symmetrised band function $g$.
As for practical reasons, it is sometimes necessary to consider a banded approximation of the symmetrised symbol function $g$. This will be addressed in detail in Appendix~\ref{Sec: symbol truncation}.

\section{Size-dependent generalised Brillouin zone}\label{sec:sharp}
The previous section was devoted to the limit of finite matrices and connecting open to periodic boundary conditions. 
At finite $n$, however, little is known on the spectral convergence of $\mathbf{T}_n(f_m)\mapsto\mathbf{T}_n(f_m\circ p)$. 
Throughout this section Assumption~\ref{ass:standing} holds and $\sigma_{\mathrm{open}}(f_m)\subset\R$. 

Let $\tau_1<\dots<\tau_n$ be the eigenvalues of $\mathbf{T}_n(f_m\circ p)$ and let $\varsigma$ be a permutation of $\{1,\dots,n\}$ minimising $\sum_{k}|\mu_{\varsigma(k)}-\tau_k|$ over all enumerations $\mu_1,\dots,\mu_n$ of $\sigma(\mathbf{T}_n(f_m))$. We write $\lambda_k:=\mu_{\varsigma(k)}$.
Their deviation is measured by the $\ell^1$ spectral distance
\begin{equation}\label{eq:dsig}
    d_\sigma\bigl(\mathbf{T}_n(f_m),\mathbf{T}_n(f_m\circ p)\bigr) :=\sum_{k=1}^{n}\bigl|\lambda_k-\tau_k\bigr| .
\end{equation}
Note that the spectral distance \eqref{eq:dsig} used in this work is an unnormalised 1-Wasserstein type distance. The fact that this distance is not normalised by $n$ makes it sensitive to even finitely many outliers: convergence with respect to this distance as $n\to \infty$ can be thought of as a strong, pointwise convergence of rate $o(n^{-1})$.

\subsection{The quantisation phases}\label{ssec:phases}
From established theory, we know that the eigenvalues of a large Toeplitz matrix are, to leading order, samples of the band function $\lambda(\cdot)$ at $n$ momenta in $(0,\pi)$. These momenta are not equispaced and they solve a quantisation rule $n\alpha+\Psi(\alpha)=\pi k$, in which the phase $\Psi$ encodes the boundary condition and depends on the symbol. The matrices $\mathbf{T}_n(f_m)$ and $\mathbf{T}_n(g)$ share the band function but not the phase, and the difference of the two phases is what we compute in this section.

The finite-size spectra are governed by two scalar phase functions on the band, one for each symbol, defined as follows. For $\lambda$ in the interior $ \sigma_{\mathrm{open}}^\circ =(A,B)$, let $\zeta = \zeta(\lambda)$ denote the central root in the upper half-plane. We split the non-central roots into the \emph{inner} group $z_1,\dots,z_{m-1}$ and the \emph{outer} group $z_{m+2},\dots,z_{2m}$, and form the unimodular quantity
\begin{equation}\label{eq:blaschke}
 B_{f_m}(\lambda):=\prod_{i=1}^{m-1}\frac{\zeta-z_i}{\bar\zeta-z_i}\ \cdot\prod_{j=m+2}^{2m}\frac{z_j-\bar\zeta}{z_j-\zeta},
\end{equation}
which has modulus one because complex roots occur in conjugate pairs (the coefficients of $z^m(f_m(z)-\lambda)$ are real for $\lambda\in\R$), and $B_{f_m}$ is well defined and continuous on $\sigma_{\mathrm{open}}^\circ$ even where inner (or outer) roots collide, being a symmetric function of each group. No factor vanishes and $B_{f_m}(\lambda)\to1$ as $\lambda$ tends to either edge.
The quantity \eqref{eq:blaschke} is precisely the phase function arising in the finite-section eigenvalue asymptotics of Bogoya-Grudsky \cite{BogoyaGrudsky2025} near simple points of the limiting set, whose proof rests on the exact determinant formula of Widom and
Baxter-Schmidt \cite{BaxterSchmidt1961},
\cite[Section~2.4]{BoettcherGrudsky2005}. Moreover, we define the phase $\varphi_f:\sigma_{\mathrm{open}}\to\R$ as the continuous branch of
\begin{equation}\label{def:phi}
    \varphi_f(\lambda) = \frac{1}{2}\operatorname{arg}B_{f_m}(\lambda)
\end{equation}
normalised by
\begin{equation}\label{eq:anchor}
    \varphi_f\bigl(\lambda(\alpha{=}0)\bigr)=0 
\end{equation}
and we define the \emph{reduced phase}
\begin{align}
 \Psi_f(\alpha) &:=m\alpha-\varphi_f\bigl(\lambda(\alpha)\bigr), \qquad \alpha\in[0,\pi] \label{eq:Psif} \\
 &= \alpha
 -\frac12\sum_{i=1}^{m-1}\arg\frac{1-z_i/\zeta}{1-z_i/\bar\zeta}   -\frac12\sum_{j=m+2}^{2m}\arg\frac{1-\bar\zeta/z_j}{1-\zeta/z_j},\quad \zeta(\alpha)=r(\alpha)e^{\i\alpha}.\label{eq:PsifWidom}
\end{align}
Here each argument is understood as a difference of principal arguments, which is continuous along the band because $|z_i/\zeta| < 1$ for the inner and $|\zeta/z_j| < 1$ for the outer group. All these terms vanish when $\zeta$ is real, so that $\Psi_f(0)=0$ and $\Psi_f(\pi) = \pi$.
On the Hermitian side $g = f_m\circ p \in \mathrm{SL}^{\beta}$, for every $\beta > 0$. We therefore use the Wiener--Hopf form of the phase, which is the standard construction in simple-loop theory \cite{BOGOYA_SLM}. Fix
$\alpha\in(0,\pi)$ and set
\begin{equation}\label{eq:Gdef}
    G_\alpha(z) := \frac{g(z)-\lambda(\alpha)}{z+z^{-1}-2\cos\alpha}, \qquad |z|=1,
\end{equation}
which removes the two zeros of $g-\lambda(\alpha)$ on $\mathbb{T}$. Then $G_\alpha$ is non-vanishing on $\mathbb{T}$ with winding number zero, hence admits a canonical Wiener--Hopf factorisation
\begin{equation}\label{eq:WH}
    G_\alpha = G^-_\alpha G^+_\alpha ,
\end{equation}
with $G^+_\alpha$ analytic and non-vanishing in the unit disc, $G^-_\alpha$ analytic and non-vanishing outside it, both normalised by $G^+_\alpha(0)=G^-_\alpha(\infty)>0$ as in \cite{BOGOYA_SLM}. We define the \emph{reduced phase} of $g$ by
\begin{equation}\label{eq:Psig}
  \Psi_g(\alpha) := \alpha + \eta_g(\alpha), \qquad \eta_g(\alpha) := \arg\frac{G^+_\alpha\bigl(e^{\i\alpha}\bigr)}
                             {G^-_\alpha\bigl(e^{\i\alpha}\bigr)}, \qquad \alpha\in(0,\pi),
\end{equation}
with the branch fixed by continuity and $\eta_g(0)=0$, and extended to $[0,\pi]$ by continuity. Equivalently, $\eta_g(\alpha)$ is the conjugate function of $\log|G_\alpha|$ evaluated at $e^{\i\alpha}$. Since $\log|G_\alpha|$ is even, its conjugate function is odd, whence $\eta_g(0) = \eta_g(\pi)= 0$, that is, $\Psi_g(0) = 0$ and $\Psi_g(\pi) = \pi$. The \emph{phase mismatch} of the transform is then immediately well defined,
\begin{equation}\label{eq:mismatch}
    \Delta\Psi:=\Psi_g-\Psi_f,\ [0,\pi]\to\R,
\end{equation}
 and continuous on $[0,\pi]$, and satisfies $\Delta\Psi(0) = \Delta\Psi(\pi)= 0$. The quantisation of both spectra is taken from the literature. In the banded but non-Hermitian case one has the following result. 
 
\begin{theorem}\cite[Theorem 2.1.]{BogoyaGrudsky2025}\label{thm:quant}
    Under Assumption~\ref{ass:standing} and $\sigma_{\mathrm{open}}\subset \R$, let us define $\varphi_f(\lambda)$ as in \eqref{def:phi}, and let
    \begin{equation}
    \Phi^{f_m}_n(\lambda):=(n+m)\alpha(\lambda)-\varphi_f(\lambda) = n\alpha(\lambda) + \Psi_f\big(\alpha(\lambda)\big).    
    \end{equation}
    Then there is $n_0$ such that for all $n\ge n_0$ the eigenvalues, sorted by the orientation of $\lambda(\cdot)$, satisfy the quantisation rule
    \begin{equation}\label{eq:quant}
     \bigl|\Phi^{f_m}_n(\lambda_k)-\pi k\bigr|= \mathcal{O}(e^{-\gamma n}), \qquad k=1,\dots,n ,
    \end{equation}
    with exactly one eigenvalue for each $k$, where $\gamma$ was defined in \eqref{eq: confluent distance}.
\end{theorem}

For the Hermitian but unbanded symbol $g = f_m \circ p \in \mathrm{SL}^\beta$ for every $\beta>0$, the following result follows directly from \cite[Theorem 2.1]{BOGOYA_SLM}.

\begin{theorem}\label{thm:quantSL}
    Let $\tau_1<\dots<\tau_n$ be the eigenvalues of $\mathbf{T}_n(g)$ and set
    \begin{equation}
        \Phi_n^{g}(\tau):=n\alpha(\tau)+\Psi_g\bigl(\alpha(\tau)\bigr), 
    \end{equation}
    with $\Psi_g$ as in \eqref{eq:Psig}. Then there is $n_0$ such that for all $n\ge n_0$ and every $q>0$,
    \begin{equation}\label{eq:quantSL}
        \bigl|\Phi_n^{g}(\tau_k)-\pi k\bigr|=o\bigl(n^{-q}\bigr), \quad k=1,\dots,n,
    \end{equation}
    uniformly in $k$, with exactly one eigenvalue for each $k$.
\end{theorem}

Let us first introduce two constants depending only on the symbol,
\begin{align}
    \Delta_2 &:=\sum_{k\in\Z}|k|\bigl( b_kb_{-k} -a_ka_{-k}\bigr), \label{eq:delta2} \\
    \hat R &:=\max\Bigl(\sum_k|a_k|,\sum_k|b_k|\Bigr), \label{eq:RHat}
\end{align}
the latter bounding the spectral radii of $\mathbf{T}_n(f_m)$ and $\mathbf{T}_n(g)$ for every $n$, since $\|\mathbf{T}_n(h)\|\le\sum_k|\hat h_k|$. We have the following trace identity.
\begin{lemma}\label{lem:trace}
For every symbol $h\in W^0$ and every $n$, $\operatorname{tr}\mathbf{T}_n(h)^2=\sum_{|k|<n}(n-|k|)\hat h_k\hat h_{-k}$. Consequently, for $n>m$,
\begin{equation}\label{eq:traceidentity}
 \operatorname{tr}\mathbf{T}_n(f_m)^2-\operatorname{tr}\mathbf{T}_n(g)^2 =\Delta_2-\mathcal{R}_n,
\end{equation}
where the tail $\mathcal{R}_n:=\sum_{|k|\ge n}\bigl(|k|-n\bigr)b_kb_{-k}$ satisfies $|\mathcal{R}_n|\le C_\eta e^{-2\eta n}$ for every $\eta<\gamma_F$.
\end{lemma}

\begin{proof}
We have $\operatorname{tr}\mathbf{T}_n(h)^2=\sum_{j,l=1}^n\hat h_{j-l}\hat h_{l-j}$, and the number of index pairs with $j-l=k$ is $n-|k|$, which gives the first identity. It converges absolutely since $h\in W^0$. For $h=f_m$ and $n>m$ the sum is finite and equals $nc_0(f_m^2)-\sum_k|k|a_ka_{-k}$. For $h=g$ the same rearrangement gives $nc_0(g^2)-\sum_k|k|b_kb_{-k}+\mathcal{R}_n$, the tail collecting the terms with $|k|\ge n$, which are absent from the finite section. Subtracting and inserting \eqref{eq:delta2} yields the first equality in \eqref{eq:traceidentity}. The bracket vanishes since $c_0(g^2)=c_0(f_m^2)$ by \cite{GBZfoundation}. Finally, by \eqref{eq: fourier coefficients of g}, $\sum_{|k|\ge n}\bigl(|k|-n\bigr)|b_k||b_{-k}|\le C_\eta e^{-2\eta n}$ for every $\eta<\gamma_F$.
\end{proof}

The following theorem gives a pointwise estimate on the eigenvalues. The proof will make use of Lemma~\ref{lem:H5} proven in the following subsection.

\begin{theorem}\label{thm:cinf}
Under Assumption~\ref{ass:standing} and $\sigma_{\mathrm{open}} \subset \R$, there is an $n_0$ such that for all $n\ge n_0$, it holds that:
\begin{itemize}
\item[(i)] For every $k\in\{1,\dots,n\}$,
\begin{equation}
 \bigl|\lambda_k-\tau_k\bigr|\ \le\ \frac{\|\lambda'\Delta\Psi\|_{L^\infty(0,\pi)}}{n} + \frac{C}{n^2}.
\end{equation}
Moreover, with $\alpha_k :=\alpha(\lambda_k)$ and uniformly in $k$, 
\begin{equation}\label{eq:secondorder}
 \lambda_k-\tau_k =\frac{\lambda'(\alpha_k)\Delta\Psi(\alpha_k)}{n} +\mathcal{O}\bigl(n^{-2}\bigr).
\end{equation}
\item[(ii)] Let us define
\begin{equation}\label{eq:cinf}
 c_\infty:=\frac1\pi\int_0^\pi\bigl|\lambda'(\alpha)\bigr| \bigl|\Delta\Psi(\alpha)\bigr|\dx\alpha = \frac1\pi\int_A^B\Big|\Delta\Psi\big(\alpha(\lambda)\big)\Big|\dx\lambda  \le \frac{B-A}{\pi}\|\Delta\Psi\|_\infty ,
\end{equation}
then the $\ell^1$ distance defined in \eqref{eq:dsig} satisfies
\begin{equation}
    \bigl|d_\sigma\bigl(\mathbf{T}_n(f_m),\mathbf{T}_n(g)\bigr) -c_\infty\bigr|\le Cn^{-1}.
\end{equation}
\item[(iii)] For every $n>m$ and every bijection $\varsigma$ of $\{1,\dots,n\}$,
\begin{equation}\label{eq:lowerbound}
 \sum_{k}\bigl|\lambda_k-\tau_{\varsigma(k)}\bigr|
 \ \ge\ \frac{\bigl|\Delta_2-\mathcal{R}_n\bigr|}{2\hat R}
 \ \ge\ \frac{|\Delta_2|-C_\eta e^{-2\eta n}}{2\hat R};
 \qquad\text{in particular}\qquad
 c_\infty\ \ge\ \frac{|\Delta_2|}{2\hat R},
\end{equation}
where $\mathcal{R}_n$ is the tail of Lemma~\ref{lem:trace} and $\eta<\gamma_F$.
\end{itemize}
\end{theorem}

\begin{proof}
We prove each assertion of the theorem separately: 

\begin{itemize}
    \item[(i)] By Theorem~\ref{thm:quant} applied to the banded symbol $f_m$ and Theorem~\ref{thm:quantSL} applied to the dense symbol $g$, both spectra are real and simple for $n\ge n_0$ and, in reduced form,
\begin{align}
    n\alpha(\lambda_k)+\Psi_f(\alpha(\lambda_k)) &=\pi k+\mathcal{O}(e^{-\gamma n})\label{eq: phase f} \\ n\alpha(\tau_k)+\Psi_g(\alpha(\tau_k)) &=\pi k+o(n^{-q}),\label{eq: phase g}
\end{align}
for any $q>0$, and for $k = 1, \dots, n$. Let us define $\delta_k:=\alpha(\lambda_k)-\alpha(\tau_k)$.
Subtracting \eqref{eq: phase f} and \eqref{eq: phase g}, and using that $\Psi_f$ and $\Psi_g = \Psi_f + \Delta \Psi$ are Lipschitz on the compact domain $[0, \pi]$ by Lemma~\ref{lem:H5} (iii), yields
\begin{align}\label{eq:alphadiff}
    n\bigl[\alpha(\lambda_k)-\alpha(\tau_k)\bigr]&=\Psi_g\big(\alpha(\tau_k)\big)-\Psi_f\big(\alpha(\lambda_k)\big) + o(n^{-q})\\
    &=\Psi_g\big(\alpha(\lambda_k)-\delta_k\big)-\Psi_f\big(\alpha(\lambda_k)\big) + o(n^{-q}) \\
    &=\Psi_g\big(\alpha(\lambda_k)\big)-\Psi_f\big(\alpha(\lambda_k)\big) + \mathcal{O}(n^{-1})\\
    &= \Delta\Psi\big(\alpha(\lambda_k)\big) + \mathcal{O}(n^{-1}),
\end{align}
where the above computations hold uniformly in $k$. By Lemma~\ref{cor:band function} we apply the mean value theorem to $\lambda(\cdot)$ on the segment between $\alpha_k=\alpha(\lambda_k)$ and $\alpha(\tau_k)$, which gives $\lambda_k-\tau_k=\lambda'(\eta_k)\delta_k$ for some $\eta_k$ with $|\eta_k-\alpha_k|\le|\delta_k|\le Cn^{-1}$. In particular because of Lemma~\ref{cor:band function} $\lambda'$ is Lipschitz on any compact domain, so $\lambda'(\eta_k)=\lambda'(\alpha_k)+\mathcal O(n^{-1})$, whence
\begin{equation}
 \lambda_k-\tau_k=\frac{\lambda'(\alpha_k)\Delta\Psi(\alpha_k)}{n} +\mathcal O\bigl(n^{-2}\bigr) \qquad\text{uniformly in }k,
\end{equation}
and taking the supremum over $\alpha_k\in(0,\pi)$ yields the first statement of (i). Near the edges the small factor $|\lambda'|$ only strengthens the pointwise bound.

\item[(ii)] Write $h:=|\lambda'||\Delta\Psi|$. By Lemma~\ref{cor:band function}, $h$ is Lipschitz on the closed interval
$[0,\pi]$. By (i), uniformly in $k$,
\begin{equation}\label{eq:dsigmasum}
  d_\sigma\bigl(\mathbf{T}_n(f_m),\mathbf{T}_n(g)\bigr)
  =\frac1n\sum_{k=1}^{n}h(\alpha_k)+\mathcal{O}\bigl(n^{-1}\bigr).
\end{equation}
Since $\Psi_f(0)=0$ and $\Psi_f(\pi)=\pi$, for $n>\|\Psi_f'\|_\infty$ the map
$\Phi_n=n\mathrm{id}+\Psi_f$ is an increasing bijection of $[0,\pi]$ onto $[0, (n+1)\pi]$, and by
\eqref{eq:quant} the nodes $\alpha_k=\Phi_n^{-1}(\pi k)+\mathcal{O}(e^{-\gamma n})$
have consecutive gaps $\pi/n+\mathcal{O}(n^{-2})$, uniformly in $k$. Hence
\eqref{eq:dsigmasum} is a Riemann sum for a Lipschitz integrand with mesh
$\pi/n$, and
\begin{equation}
  \frac1n\sum_{k=1}^{n}h(\alpha_k)=\frac1\pi\int_0^\pi h\dx\alpha
  +\mathcal{O}\bigl(n^{-1}\bigr),
\end{equation}
which gives $\bigl|d_\sigma-c_\infty\bigr|\le Cn^{-1}$. The equality of the two integrals in \eqref{eq:cinf} is the substitution $\dx\lambda=\lambda'(\alpha)\dx\alpha$, and the final inequality uses $\int_0^\pi|\lambda'|\dx\alpha=B-A$.

\item[(iii)] By Lemma~\ref{lem:trace} and
$\sum_k\lambda_k^2=\operatorname{tr}\mathbf{T}_n(f_m)^2$,
$\sum_k\tau_k^2=\operatorname{tr}\mathbf{T}_n(g)^2$,
\begin{equation}
    \Bigl|\sum_k\bigl(\lambda_k^2-\tau_{\varsigma(k)}^2\bigr)\Bigr| = \bigl|\Delta_2-\mathcal{R}_n\bigr|, \qquad \bigl|\lambda_k^2-\tau_{\varsigma(k)}^2\bigr|\le 2\hat R\bigl|\lambda_k-\tau_{\varsigma(k)}\bigr|,
\end{equation}
which yields the result.
\end{itemize}
\end{proof}

Theorem~\ref{thm:cinf} gives an upper and a lower bound for the eigenvalue error of the transform. Pointwise, every eigenvalue is reproduced to $\mathcal{O}(1/n)$, with the exact first-order error $\lambda'\Delta\Psi$. Globally, the $\ell^1$ distance does \emph{not} tend to zero but to the constant $c_\infty$, which is positive unless $\Delta\Psi \equiv 0$ (as happens for $m = 1$), and therefore measures how far the transform is from a similarity. An improvement on the spectral distance will be presented in the following section.

\subsection{The first-order correction}\label{sec:firstorder}

Theorem~\ref{thm:cinf} shows that the symmetrisation $\mathbf{T}_n(f_m)\mapsto\mathbf{T}_n(g)$ is exact only to leading order in the small variable $1/n$ as $n \to \infty$.  
The starting point is the structure of the two quantisation rules \eqref{eq: phase f} and \eqref{eq: phase g}. Both spectra are read off from the same band function $\lambda(\cdot)$ as by construction $\lambda(\alpha)=g(e^{i\alpha})$. What differs is only the reduced phase, $\Psi_f$ versus $\Psi_g$ given by \eqref{eq:Psif} and \eqref{eq:Psig} respectively. Their difference is the phase mismatch $\Delta\Psi(\alpha):= \Psi_g(\alpha)-\Psi_f(\alpha)$ of \eqref{eq:mismatch}. 
Subtracting \eqref{eq:PsifWidom} from \eqref{eq:Psig}, with $z_1,\dots,z_{2m}$ as in Theorem~\ref{thm: schmidt spitzer magnitude} and $\zeta=z_{m}$ the central root in the upper half-plane, the phase mismatch \eqref{eq:mismatch} reads
\begin{equation}\label{eq:epsexplicit}
  \Delta\Psi(\alpha)= \eta_g(\alpha) +\frac12\sum_{i=1}^{m-1}\arg\left(\frac{1-z_i/\zeta}{1-z_i/\overline\zeta}\right) +\frac12\sum_{j=m+2}^{2m}\arg\left(\frac{1-\bar\zeta/z_j}{1-\zeta/z_j}\right),
\end{equation}
evaluated at $\lambda=\lambda(\alpha)$, with the branch fixed by $\Delta\Psi(0)=0$. 
The heuristic is to eliminate the phase mismatch by adjusting the parametrisation of the GBZ. The discrepancy in \eqref{eq:alphadiff} motivates us to define the first order phase-corrected parametrisation of the GBZ, 
\begin{equation}\label{eq:pn}
  p_n(e^{\i\alpha}):=p\Big(e^{\i\left(\alpha+\Delta\Psi(\alpha)/n\right)}\Big).
\end{equation}
In particular, the reparametrisation of the GBZ is now size-dependent and the corresponding symbol is given by
\begin{equation}\label{eq: phase-corrected symbol}
    g_n(e^{\i\alpha}):=f_m\circ p_n(e^{\i\alpha}) =\lambda\Big(\alpha+\frac{\Delta\Psi(\alpha)}{n}\Big).
\end{equation}

As the next result shows, the phase-correction to the parametrisation of the generalised Brillouin zone leads to a perturbation of the Toeplitz matrix.
\begin{lemma}\label{lem:struct}
    Under Assumption~\ref{ass:standing} and $\sigma_{\mathrm{open}} \subset \R$, the function $\Delta\Psi$ is continuous on $[0,\pi]$ and $\Delta\Psi(0)=\Delta\Psi(\pi)=0$. Moreover, its odd $2\pi$-periodic extension is continuous on $\mathbb{R}$. Consequently $g_n$ is real-valued and even, so $\mathbf{T}_n(g_n)$ is a real symmetric Toeplitz matrix with,
    \begin{equation}\label{eq:additivematrix}
      \mathbf{T}_n(g_n)=\mathbf{T}_n(g)+\frac1n\mathbf{T}_n(\lambda'\Delta\Psi)+E_n, \quad \|E_n\|_2\le Cn^{-2}.
    \end{equation}
\end{lemma}

\begin{proof}
Since $\Psi_f$ and $\Psi_g$ agree at $0$ and $\pi$ we have $\Delta\Psi(0)=\Delta\Psi( \pi)=0$, so the odd $2\pi$-periodic extension is continuous. Since $\Delta\Psi$ is odd and $\lambda$ is even, $\alpha\mapsto\alpha+\Delta\Psi(\alpha)/n$ is odd and $g_n(e^{-\i\alpha})=\lambda(-\alpha-\Delta\Psi(\alpha)/n)=g_n(e^{\i\alpha})$, hence $\widehat{(g_n)}_k=\widehat{(g_n)}_{-k}\in\mathbb{R}$. To show the matrix expansion in \eqref{eq:additivematrix} we may proceed as follows. Taylor expanding the phase-corrected symbol given by \eqref{eq: phase-corrected symbol} gives,
\begin{equation}\label{eq:symbolexpansion}
  g_n  =  g  +  \frac1n\lambda'\Delta\Psi  +  R_n, \qquad \|R_n\|_{L^\infty(\mathbb{T})}\le\frac{\|\lambda''\|_\infty\|\Delta\Psi\|_\infty^2}{2n^{2}} \leq Cn^{-2},
\end{equation}
where we used $\lambda(\alpha)=g(e^{\i\alpha})$. Since $h\mapsto\mathbf{T}_n(h)$ is linear and $\mathbf{T}_n(h)=P_nL(h)P_n$ is a compression of the Laurent operator, so that $\|\mathbf{T}_n(h)\|_2\le\|h\|_{L^\infty(\mathbb{T})}$, this transfers to the matrices giving \eqref{eq:additivematrix}.
The function $\lambda'\Delta\Psi$ is real and even, being the product of the odd functions $\lambda'$ and $\Delta\Psi$, so $\mathbf{T}_n(\lambda'\Delta\Psi)$ is again real symmetric. Both matrices in \eqref{eq:additivematrix} are real symmetric, which completes the proof.
\end{proof}

The correction (\ref{eq:pn}) is tangential. For $n > \lVert \Delta\Psi'\rVert_\infty$ it composes $p$ with a diffeomorphism of the torus and therefore leaves the image $p(\mathbb{T})=\mathbf{Z}(f_m)$ unchanged, altering only its parametrisation.
Applying \Cref{thm:quantSL} to the corrected symbol requires a uniformity statement, since $g_n$ varies with $n$.
Throughout, $\Psi_{g_n}$ is defined by \eqref{eq:Gdef}--\eqref{eq:Psig} with $g$ replaced by $g_n$, that is, with $\lambda(\alpha)$ replaced by $g_n(e^{\i\alpha})=\lambda(\chi_n(\alpha))$, where $\chi_n:=\mathrm{id}+\Delta\Psi/n$.

\begin{lemma}\label{lem:H5}
Suppose Assumption~\ref{ass:standing} holds and $\sigma_{\mathrm{open}}\subset\R$. Then there is $n_1\in\N$ such that:
\begin{enumerate}
\item[(i)] the quantisation rule of \Cref{thm:quantSL} holds for the family $\{g_n\}_{n\ge n_1}$, with
  remainder $C_qn^{-q}$ where $C_q$ is independent of $n$;
\item[(ii)] $\|\Psi_{g_n}-\Psi_g\|_{L^\infty([0,\pi])}\le Cn^{-1}$;
\item[(iii)] there is $\eta>0$ such that $\Psi_f-\mathrm{id}$ and $\Delta\Psi$ extend analytically and
  boundedly to $S_\eta=\{|\Im\alpha|<\eta\}$, where they are odd and $2\pi$-periodic; in particular both
  are Lipschitz on $[0,\pi]$.
\end{enumerate}
\end{lemma}

\begin{proof}
By \Cref{thm: exact characterisation}, $\zeta(\alpha)=r(\alpha)e^{\i\alpha}$ and $\lambda=f_m\circ\zeta$ are analytic and bounded on $S_{4\eta}$ for some $\eta>0$.

(iii) For $\alpha\in S_{2\eta}$ the roots of $z^m(f_m(z)-\lambda(\alpha))$ are $\zeta(\pm\alpha)$, the inner roots $z_1,\dots,z_{m-1}$ and the outer roots $z_{m+2},\dots,z_{2m}$. By \eqref{eq:strictsep} and compactness, after shrinking $\eta$ we have $|z_i/\zeta|\le e^{-\gamma/2}$ and $|\zeta/z_j|\le e^{-\gamma/2}$ there. Symmetric functions of each group are analytic. Here $\operatorname{Log}$ is the principal branch of the logarithm, which is analytic on $\operatorname{Re} w>0$ and hence on all arguments below, since $|z_i/\zeta|,|\zeta/z_j|<1$. Thus
\begin{equation}
  \Theta(\alpha):=\sum_{i=1}^{m-1}\operatorname{Log}\Bigl(1-\frac{z_i}{\zeta}\Bigr)
  -\sum_{j=m+2}^{2m}\operatorname{Log}\Bigl(1-\frac{\zeta}{z_j}\Bigr)
  =-\sum_{p\ge1}\frac1p\Bigl(\sum_i\bigl(\tfrac{z_i}{\zeta}\bigr)^p-\sum_j\bigl(\tfrac{\zeta}{z_j}\bigr)^p\Bigr)
\end{equation}

is analytic and bounded on $S_{2\eta}$. For real $\alpha$ we have $\zeta(-\alpha)=\bar\zeta(\alpha)$, so $\overline{\Theta(\alpha)}=\Theta(-\alpha)$, and \eqref{eq:PsifWidom} reads $\Psi_f(\alpha)-\alpha=\tfrac{\i}{2}\bigl(\Theta(\alpha)-\Theta(-\alpha)\bigr)$. This gives the odd, periodic extension of $\Psi_f-\mathrm{id}$.

For $\Psi_g-\mathrm{id}=\eta_g$, write $\lambda=\Lambda\circ\cos$. Then $G_\alpha(e^{\i\theta})=\tfrac12\int_0^1\Lambda'\bigl((1-t)\cos\alpha+t\cos\theta\bigr)\dx t$, and $\Lambda'(\cos\theta)=-\lambda'(\theta)/\sin\theta$ has a fixed sign $\sigma$ and satisfies $\sigma\Lambda'\ge 2c_1>0$ by \Cref{cor:band function}. Hence, after shrinking $\eta$, the function $(\theta,\alpha)\mapsto G_\alpha(e^{\i\theta})$ is analytic on $S_{2\eta}^2$ with $\operatorname{Re}\sigma G_\alpha\ge c_1$, and $\log(\sigma G_\alpha(e^{\i\theta})) =\sum_k c_k(\alpha)e^{\i k\theta}$ is analytic and bounded by some $M$ there. Then $|c_k(\alpha)|\le Me^{-2\eta|k|}$, and
\begin{equation}
  \eta_g(\alpha)=-\i\sum_{k}\operatorname{sgn}(k)c_k(\alpha)e^{\i k\alpha}
\end{equation}
converges uniformly on $S_\eta$, which gives the analytic extension. Since $G_\alpha$ is even in $\theta$ and in $\alpha$, $\eta_g$ is odd. Hence $\Delta\Psi=\eta_g-(\Psi_f-\mathrm{id})$ has the claimed properties, and Lipschitz continuity on $[0,\pi]$ follows from Cauchy's estimate.

(i) By (iii), for $n\ge n_1$ the map $\chi_n$ sends $S_{\eta/2}$ into $S_\eta$ and satisfies $\tfrac12\le\chi_n'\le\tfrac32$ on $\R$. Hence $g_n$ is analytic on $S_{\eta/2}$ with $|g_n|\le\|\lambda\|_{2\eta}$ there, so $|\widehat{(g_n)}_k|\le\|\lambda\|_{2\eta}e^{-\eta|k|/2}$ and $\|g_n\|_{W^\beta}$ is bounded uniformly in $n$ for every $\beta>0$. Since $\chi_n$ is odd and fixes $0$ and $\pi$, the bound $|\lambda'(\alpha)|\ge c|\sin\alpha|$ from \Cref{cor:band function} gives $|\tfrac{\dx}{\dx\alpha}g_n(e^{\i\alpha})|\ge c'|\sin\alpha|$. Moreover $\tfrac{\dx^2}{\dx\alpha^2}g_n(e^{\i\alpha})\big|_{\alpha_e}=\lambda''(\alpha_e)\chi_n'(\alpha_e)^2$ for $\alpha_e\in\{0,\pi\}$, which is bounded away from $0$. Thus $g_n\in\mathrm{SL}^\beta$ uniformly in $n$ for every $\beta>0$, and the remainder in \cite[Theorem~2.1]{BOGOYA_SLM} depends on the symbol only through these bounds.

(ii) Let $G^{(n)}_\alpha$ denote \eqref{eq:Gdef} for $g_n$. Then $G^{(n)}_\alpha(e^{\i\theta})-G_\alpha(e^{\i\theta})=\frac{d(\theta)-d(\alpha)}{2(\cos\theta-\cos\alpha)}$ with $d(\alpha):=g_n(e^{\i\alpha})-g(e^{\i\alpha})$, which is even and satisfies $\|d\|_{\eta/2}\le\|\lambda'\|_\eta\|\Delta\Psi\|_\eta/n$. Writing $d=D\circ\cos$, Cauchy's formula for the divided difference of $D$ on the ellipse $\cos(S_{\eta/2})$ bounds this difference by $C/n$ on $S_{\eta/4}^2$. Hence, for $n\ge n_1$, $\operatorname{Re}\sigma G^{(n)}_\alpha\ge c_1/2$ and the logarithms differ by at most $C/n$ there. The coefficient bound and series of (iii), applied on $S_{\eta/4}^2$, give $\|\Psi_{g_n}-\Psi_g\|_{\eta/8}\le C/n$, and in particular (ii).
\end{proof}

We are now ready to state the main theorem of this section which establishes that the phase-corrected symbol achieves a higher pointwise eigenvalue convergence.
\begin{theorem}\label{thm:corrected}
    Suppose that Assumption~\ref{ass:standing} holds and that $\sigma_{\mathrm{open}}\subset \R$. Let $\tau^{(n)}_1\le\dots\le\tau^{(n)}_n$ be the eigenvalues of the real symmetric matrix $\mathbf{T}_n(g_n)$, where the symbol function $g_n$ was defined in \eqref{eq: phase-corrected symbol}. Then there is $n_0$ such that for all $n\ge n_0$ it holds that:
    \begin{enumerate}
        \item[(i)] There exists a constant $C=C(f_m)$ independent of $k$, such that
        \begin{equation}\label{eq:pointwisecorr}
          \big|\lambda_k-\tau^{(n)}_k\big|\le \frac{C}{n^{2}} \qquad\text{uniformly in }k\in\{1,\dots,n\};
        \end{equation}
        \item[(ii)] The $\ell^1$ spectral distance introduced in \eqref{eq:dsig} satisfies,        \begin{equation}\label{eq:l1corr}
          d_\sigma\big(\mathbf{T}_n(f_m),\mathbf{T}_n(g_n)\big)\le Cn^{-1}.
        \end{equation}
    \end{enumerate}
\end{theorem}

\begin{proof}
    By \Cref{thm:quant} for $f_m$ and by Theorem~\ref{thm:quantSL} together with Lemma~\ref{lem:H5}(i) for $g_n$, both rules hold for $n\ge n_0$ with the same uniform constants: $\lambda_k=\lambda(\theta_k)$ and $\tau^{(n)}_k=\lambda(\sigma_k)$ with $\sigma_k:=\beta_k+\Delta\Psi(\beta_k)/n$, where
    \begin{align} 
      n\theta_k+\Psi_f(\theta_k)&=\pi k+\mathcal{O}(e^{-\gamma n}),\label{eq: rule f} \\
      n\beta_k+\Psi_{g_n}(\beta_k)&=\pi k+o(n^{-q}), \label{eq: rule gn}
    \end{align}
    for any $q>0$, with exactly one solution per $k$, both enumerations sorted by the same monotone $\lambda(\cdot)$. Substituting $\sigma_k$ into the quantisation rule of $f_m$ given in \eqref{eq: rule f} yields
    \begin{equation}
        n \sigma_k + \Psi_f(\sigma_k) = n \beta_k + \Delta\Psi(\beta_k) + \Psi_f\Bigl(\beta_k + \frac{\Delta\Psi(\beta_k)}{n}\Bigr).
    \end{equation}
    Replacing $n\beta_k$ by means of the quantisation rule \eqref{eq: rule gn} and then substituting the definition of $\Delta\Psi$ from \eqref{eq:mismatch}, we obtain
    \begin{align}
        n\sigma_k+\Psi_f(\sigma_k)
        &=\pi k+o(n^{-q})
          -\Psi_{g_n}(\beta_k)+\Delta\Psi(\beta_k)
          +\Psi_f\Bigl(\beta_k+\frac{\Delta\Psi(\beta_k)}{n}\Bigr)\\
        &=\pi k+o(n^{-q})
          -\Psi_{g_n}(\beta_k)+\Psi_g(\beta_k)-\Psi_f(\beta_k)
          +\Psi_f\Bigl(\beta_k+\frac{\Delta\Psi(\beta_k)}{n}\Bigr)\\
        &=\pi k+o(n^{-q})
          +\underbrace{\bigl[\Psi_g-\Psi_{g_n}\bigr](\beta_k)
          +\bigl[\Psi_f(\cdot+\Delta\Psi/n)-\Psi_f\bigr](\beta_k)}_{=O(n^{-1})
      \text{ by Lemma~\ref{lem:H5}(ii),(iii)}}.
        \label{eq: sigma residual}
    \end{align}
    Subtracting \eqref{eq: rule f} from \eqref{eq: sigma residual} and using that $\Psi_f$ is Lipschitz with constant $L$ by Lemma~\ref{lem:H5}(iii),
    \begin{equation*}
      (n-L)|\sigma_k-\theta_k|\le Cn^{-1}+o(n^{-q}),
    \end{equation*}
    so $|\sigma_k-\theta_k|\le Cn^{-2}$ for $n\ge n_0$, uniformly in $k$. Since $\lambda$ is Lipschitz on $[0,\pi]$ by Lemma~\ref{cor:band function}, 
    \begin{align}
       |\lambda_k-\tau^{(n)}_k|&\le\|\lambda'\|_\infty|\theta_k-\sigma_k| \\
       &\le C\|\lambda'\|_\infty n^{-2},
    \end{align}
    which is \eqref{eq:pointwisecorr}. Summing over $k$ gives
    \eqref{eq:l1corr}.
\end{proof}

Note that with the first-order phase correction taken into account, the spectral distance \eqref{eq:dsig} between $\mathbf{T}_n(f_m)$ and $\mathbf{T}_n(g_n)$ tends to $0$ as $n\to \infty$ with respect to the metric introduced in \eqref{eq:dsig}.
The spectral convergence claimed by \eqref{eq:l1corr} is numerically verified in Figure~\ref{Fig: Second Order correction}. Here, we consider a concrete example of a Toeplitz matrix and observe the convergence rate $\mathcal{O}(n^{-1})$ as predicted by Theorem~\ref{thm:corrected}. Also shown in Figure~\ref{Fig: Second Order correction} is a second-order correction, which will be introduced in the following section.

\subsection{Higher-order corrections}\label{ssec:higher}

The first-order correction presented in the previous section reparametrises the polar curve by the fixed profile $\Delta\Psi/n$. The mechanism is iterative, that is, the corrected curve can itself be corrected. The reason is that the eigenvalue error is controlled by a computable function on the band, namely the residual, which measures the quantisation difference between $g_n$ and
$f_m$. One further evaluation of a reduced phase divides the residual by $n$. Throughout, $\eta$ is as in \Cref{lem:H5}(iii) and $\mathcal{E}$ denotes the odd, $2\pi$-periodic functions that are analytic and bounded on the strip $S_\eta$, with $\|\cdot\|_\eta:=\sup_{S_\eta}|\cdot|$. By \Cref{lem:H5}(iii) we have $\Delta\Psi\in\mathcal{E}$ with $\|\Delta\Psi\|_\eta<\infty$, and shrinking $\eta$ if necessary, $\lambda$ and $\Psi_f$ are analytic and bounded on $S_\eta$ as long as $n > \lVert \tilde \varepsilon\rVert /\eta$.
For a profile $\tilde\varepsilon\in\mathcal{E}$ and $n>\|\tilde\varepsilon\|_\eta/\eta$, so that $\mathrm{id} + \tilde\varepsilon/n$ maps $S_\eta$ into $S_{2\eta}$, the \emph{corrected symbol} is given by
\begin{equation}\label{eq:corrsym}
    g_{\tilde\varepsilon}(e^{\i\beta}):=\lambda\Bigl(\beta+\frac{\tilde\varepsilon(\beta)}{n}\Bigr),
\end{equation}
whereas the remaining phase difference is captured by the \emph{residual} which is defined as
\begin{equation}\label{eq:residualcorrected}
    \rho[\tilde\varepsilon]:=\tilde\varepsilon-\Psi_{g_{\tilde\varepsilon}} +\Psi_f\circ\Bigl(\mathrm{id}+\frac{\tilde\varepsilon}{n}\Bigr).
\end{equation}
In particular, $g_{\tilde{\varepsilon}}$ is real and even so the associated Toeplitz matrix $\mathbf{T}_n(g_{\tilde\varepsilon})$ is real symmetric. By the argument of Lemma~\ref{lem:H5}(i), \Cref{thm:quantSL} applies to $g_{\tilde\varepsilon}$ uniformly on bounded sets of $\mathcal{E}$ and the $k$-th eigenvalue of $\mathbf{T}_n(g_{\tilde\varepsilon})$ is quantised as follows,
\begin{equation}
    \tau_k=\lambda(\sigma_k), \quad \sigma_k=\beta_k+\tilde\varepsilon(\beta_k)/n,
\end{equation}
and
\begin{equation}
    n\beta_k+\Psi_{g_{\tilde\varepsilon}}(\beta_k)=\pi k+o(n^{-q}),
\end{equation}
for every $q>0$, uniformly in $k$ and  in $\tilde\varepsilon \in \mathcal{E}$. Eliminating $n\beta_k$ in the rule of $f_m$ evaluated at $\sigma_k$ gives the exact identity
\begin{align}\label{eq:resder}
  n\sigma_k+\Psi_f(\sigma_k)
  &=\pi k+o(n^{-q}) +\tilde\varepsilon(\beta_k)-\Psi_{g_{\tilde\varepsilon}}(\beta_k) +\Psi_f\bigl(\beta_k+\tfrac{\tilde\varepsilon(\beta_k)}{n}\bigr)\\
  &=\pi k+ o(n^{-q}) + \rho[\tilde\varepsilon](\beta_k).
\end{align}
This yields the following pointwise eigenvalue estimate, which shows that the error is controlled, to leading
order, by the norm of the residual.

\begin{lemma}\label{lem:residual-eps}
Let $\tilde\varepsilon\in\mathcal{E}$ and let $n>\|\Psi_f'\|_\infty$ be large enough for both quantisation rules. Then, for $k=1,\dots,n$, with $\lambda_k$ the eigenvalues of $\mathbf{T}_n(f_m)$,
\begin{equation}\label{eq:resbound}
  \bigl|\lambda_k-\tau_k\bigr|
  \le\frac{\|\lambda'\|_\infty}{n-\|\Psi_f'\|_\infty}
  \Bigl(\bigl\|\rho[\tilde\varepsilon]\bigr\|_{L^\infty(0,\pi)}+Ce^{-\gamma n}+o(n^{-q})\Bigr).
\end{equation}
\end{lemma}

\begin{proof}
    Subtracting the rule $n\theta_k+\Psi_f(\theta_k)=\pi k+\mathcal{O}(e^{-\gamma n})$ from \eqref{eq:resder} and using the mean value theorem to obtain the bound
    \begin{equation}
        |\Psi_f(\sigma_k)-\Psi_f(\theta_k)|\le\|\Psi_f'\|_\infty|\sigma_k-\theta_k|,    
    \end{equation}
    gives
    \begin{equation}
      (n-\|\Psi_f'\|_\infty)|\sigma_k-\theta_k|\leq\bigl\|\rho[\tilde\varepsilon]\bigr\|_{L^\infty(0,\pi)}+Ce^{-\gamma n}+o(n^{-q}),
    \end{equation}
    and then applying the mean value theorem once again on
    \begin{equation}
        |\lambda_k-\tau_k|=|\lambda(\theta_k)-\lambda(\sigma_k)|\leq\|\lambda'\|_\infty|\sigma_k-\theta_k|    
    \end{equation}
    yields the desired result.
\end{proof}

By Lemma~\ref{lem:residual-eps}, the eigenvalue error is controlled by $\|\rho[\tilde\varepsilon]\|_\infty$. We now show that each iteration decreases
the norm of the residual by one additional order in $1/n$.

\begin{theorem}\label{thm:higher}
Define the profiles at the $K$-th iteration
\begin{equation}\label{eq:picard-eps}
  \varepsilon^{(0)}:=0,\qquad
  \varepsilon^{(K+1)}:=\Psi_{g_{\varepsilon^{(K)}}}-\Psi_f\circ\Bigl(\mathrm{id}+\frac{\varepsilon^{(K)}}{n}\Bigr),
  \qquad K=0,1,2,\dots
\end{equation}
Then there is $R>0$ depending on $f_m$ only and, for every $K\in\N$, constants $n_K,L_K>0$ depending on $f_m$ and $K$ only, such that for $n\ge n_K$ the profiles $\varepsilon^{(0)},\dots,\varepsilon^{(K)}$ are odd, $2\pi$-periodic, analytic and bounded by $R$ on the strip $S_{\eta/2}$, and
\begin{equation}\label{eq:respicard}
 \bigl\|\rho\bigl[\varepsilon^{(K)}\bigr]\bigr\|_{\eta/2}\le\Bigl(\frac{L_K}{n}\Bigr)^{K}\|\Delta\Psi\|_\eta.
\end{equation}
Consequently, if $\tau^{(n,K)}_1\le\dots\le\tau^{(n,K)}_n$ are the eigenvalues of the real symmetric matrix $\mathbf{T}_n(g_{\varepsilon^{(K)}})$, there is $C_K$ depending on $f_m$ and $K$ only such that, enlarging $n_K$ if necessary, for $n\ge n_K$
\begin{equation}\label{eq:higher}
 \bigl|\lambda_k-\tau^{(n,K)}_k\bigr|\le\frac{C_K}{n^{K+1}}\ \text{ uniformly in }k,
 \qquad
 d_\sigma\bigl(\mathbf{T}_n(f_m),\mathbf{T}_n(g_{\varepsilon^{(K)}})\bigr)\le\frac{C_K}{n^{K}}.
\end{equation}
\end{theorem}

\begin{proof}
Write $\mathcal{F}_n[\tilde\varepsilon]:=\Psi_{g_{\tilde\varepsilon}}-\Psi_f\circ(\mathrm{id}+\tilde\varepsilon/n)$, so that $\rho[\tilde\varepsilon]=\tilde\varepsilon-\mathcal{F}_n[\tilde\varepsilon]$ and the iteration reads $\varepsilon^{(K+1)}=\mathcal{F}_n[\varepsilon^{(K)}]$. For $0<s\le\eta$ let $\mathcal{E}_s$ be the odd, $2\pi$-periodic functions analytic and bounded on $S_s$, normed by $\|\cdot\|_s$ and note $\mathcal{E}_s\subset\mathcal{E}_{s'}$ with $\|\cdot\|_{s'}\le\|\cdot\|_s$ for $s'<s$. Put $R:=2\|\Delta\Psi\|_\eta$ and
\begin{equation}
  \eta_j:=\eta\bigl(\tfrac12+2^{-j-1}\bigr),\qquad j\ge0,
\end{equation}
so $\eta_0=\eta$, $\eta_j\downarrow\eta/2$ and $\eta_j-\eta_{j+1}=\eta2^{-j-2}$.

The map $\mathcal{F}_n$ preserves oddness, but it is not a self-map of a single $\mathcal{E}_s$. The Wiener--Hopf factorisation is bounded only on strictly smaller strips. Let $0<s'<s\le\eta$ and $u_1,u_2\in\mathcal{E}_s$ with $\|u_i\|_s\le R$. For $n\ge R/\eta$ the maps $\mathrm{id}+u_i/n$ send $S_s$ into $S_{2\eta}$, so $g_{u_i}$ is analytic and bounded on $S_s$ with $\|g_{u_1}-g_{u_2}\|_s\le\|\lambda'\|_{2\eta}\|u_1-u_2\|_s/n$. With the stability estimate of Lemma~\ref{lem:H5}(ii) and the Lipschitz continuity of $\Psi_f$ on $S_\eta$ this gives a constant $L$ depending on $f_m$ only, such that
\begin{equation}\label{eq:Flip}
  \bigl\|\mathcal{F}_n[u_1]-\mathcal{F}_n[u_2]\bigr\|_{s'}\le\frac{L}{n(s-s')}\|u_1-u_2\|_s,
  \qquad
  \bigl\|\mathcal{F}_n[u]\bigr\|_{s'}\le\|\Delta\Psi\|_\eta+\frac{L}{n(s-s')}\|u\|_s .
\end{equation}
Write $L_j:=L/(\eta_j-\eta_{j+1})=2^{j+2}L/\eta$ for the constant at the $j$-th scale.

Fix $K$ and let $n\ge n_K:=\max(R/\eta,2L_K)$. Then $L_j/n\le\tfrac12$ for $j\le K$, so the second estimate in \eqref{eq:Flip}, together with $\varepsilon^{(0)}=0$, gives inductively
\begin{equation}  \bigl\|\varepsilon^{(j)}\bigr\|_{\eta_j}\le\|\Delta\Psi\|_\eta+\tfrac12R=R,\qquad j=0,\dots,K .
\end{equation}
In particular each $\varepsilon^{(j)}$, $j\le K$, is bounded by $R$ on $S_{\eta/2}$.

As before define the residual recursively,
\begin{equation}\label{eq: residual identity}
  \rho\bigl[\varepsilon^{(K+1)}\bigr]=\varepsilon^{(K+1)}-\mathcal{F}_n\bigl[\varepsilon^{(K+1)}\bigr]
  =\mathcal{F}_n\bigl[\varepsilon^{(K)}\bigr]-\mathcal{F}_n\bigl[\varepsilon^{(K+1)}\bigr],
\end{equation}
and the increment of the arguments is the previous residual, $\varepsilon^{(K)}-\varepsilon^{(K+1)}=\rho[\varepsilon^{(K)}]$. Both arguments lie in $\mathcal{E}_{\eta_{K+1}}$ with norm at most $R$. Applying the first estimate of \eqref{eq:Flip} at the scale $\eta_{K+1}\to\eta_{K+2}$,
\begin{equation}  \bigl\|\rho\bigl[\varepsilon^{(K+1)}\bigr]\bigr\|_{\eta_{K+2}}\le\frac{L_{K+1}}{n}\bigl\|\rho\bigl[\varepsilon^{(K)}\bigr]\bigr\|_{\eta_{K+1}} .
\end{equation}
Since $\rho[\varepsilon^{(0)}]=-\Delta\Psi$, induction gives $\|\rho[\varepsilon^{(K)}]\|_{\eta_{K+1}}\le n^{-K}\bigl(\prod_{j=1}^{K}L_j\bigr)\|\Delta\Psi\|_\eta$, and \eqref{eq:respicard} follows with $L_K:=\bigl(\prod_{j=1}^{K}L_j\bigr)^{1/K}=\tfrac{4L}{\eta}2^{(K+1)/2}$, because $\eta_{K+1}\ge\eta/2$. Inserting \eqref{eq:respicard} into \eqref{eq:resbound} yields the pointwise bound in \eqref{eq:higher}, summing over $k$ yields the second.
\end{proof}

The cases $K=0,1$ are \Cref{thm:cinf}(i) and \Cref{thm:corrected} respectively. The previous theorem motivates to write,
\begin{equation}
    p_n^{(k)}(e^{\i\alpha})=p(e^{\i(\alpha + \varepsilon^{(K)}/n})
\end{equation}
for the $K$-th order corrected parametrisation of the GBZ, so that $g_{\varepsilon^{(K)}} = f_m\circ p_n^{(K)}$
The results will be numerically illustrated in Figure~\ref{Fig: Second Order correction}.

\begin{figure}[h]
    \centering
    \subfloat[][Spectral convergence up to the second order phase correction.]%
    {\includegraphics[width=0.48\linewidth]{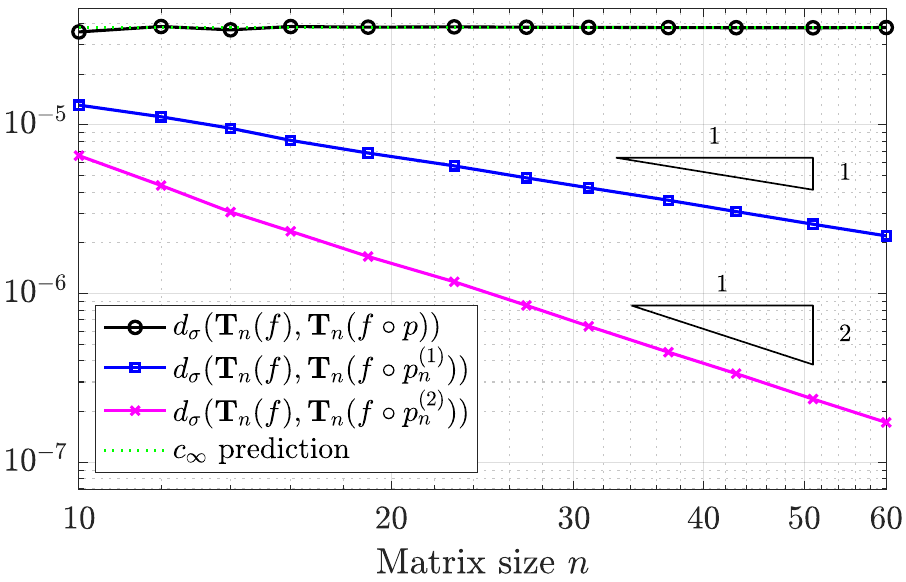}}\quad
    \subfloat[][Spectral convergence up to the second order phase correction.]%
    {\includegraphics[width=0.48\linewidth]{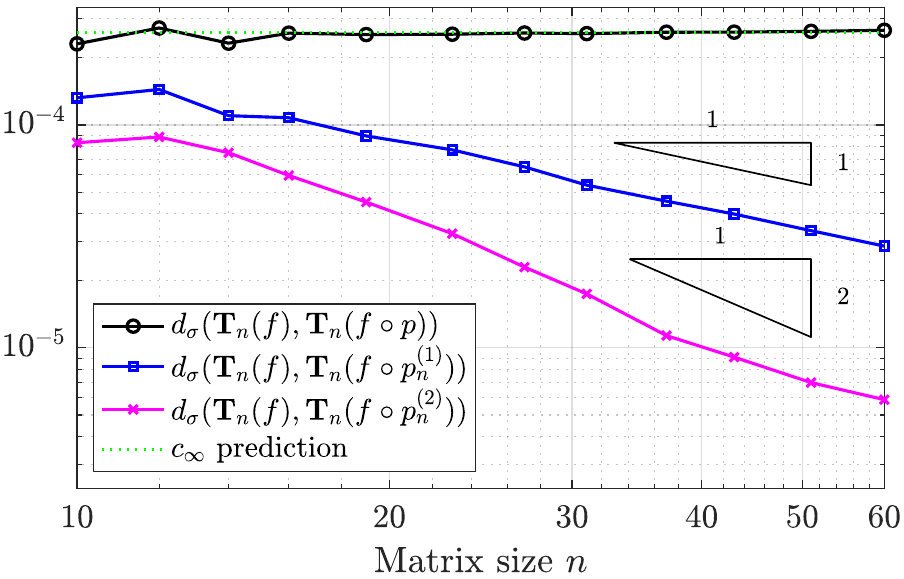}} \\
    \subfloat[][Phase shift captured by the second order expansion.]%
    {\includegraphics[width=0.48\linewidth]{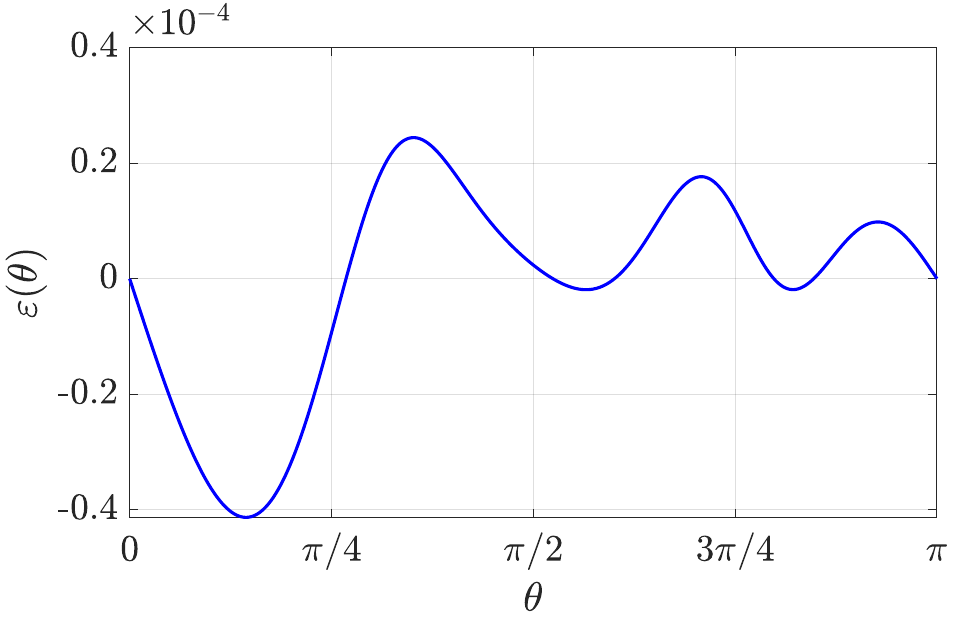}}\quad
    \subfloat[][Phase shift captured by the second order expansion.]%
    {\includegraphics[width=0.48\linewidth]{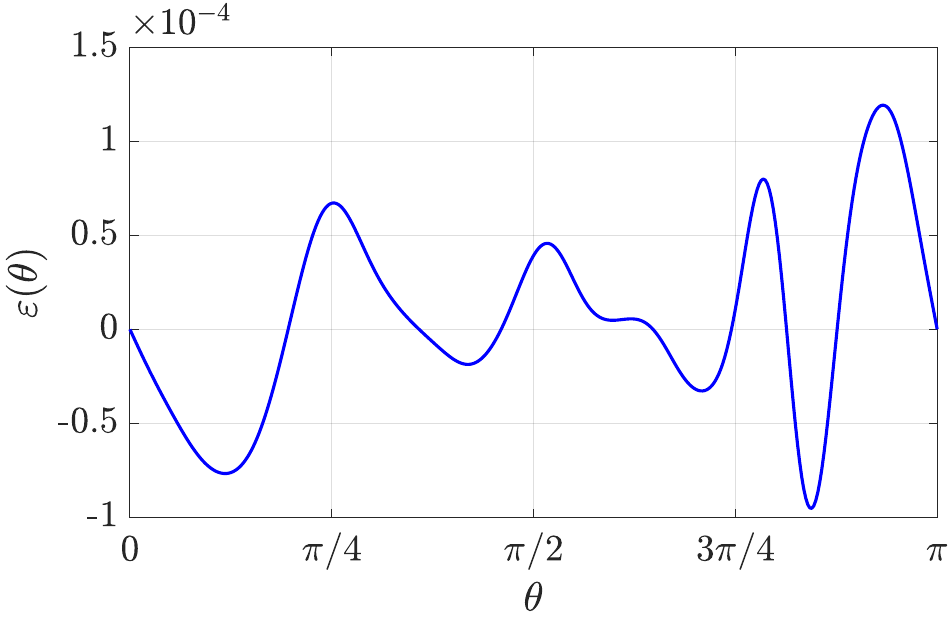}}
    \caption{We consider the spectral convergence of the symmetrisation induced by the parametrisation $p$ of the GBZ, along with the first and second order corrections defined in Section~\ref{sec:firstorder} and Section~\ref{ssec:higher}, respectively. Here, we take a Toeplitz matrix with algebraic off-diagonal decay such that $a_k = |k|^{-\kappa},~ k > 0,$ $a_k = |k|^{-\rho},~k < 0$ and $a_0 = 1$; (A) and (C): computation performed for $m=5, \kappa = 4.5$ and $\rho = 6.8$; (B) and (D): computation performed for $m=7$, with $\kappa = 3.5$ and $\rho = 4.8$.}
    \label{Fig: Second Order correction}
\end{figure}

\section{Non-Hermitian Condensed matter physics}\label{Sec: Physical Applications}
    
The generalised Brillouin zone is defined through the limit $n\to\infty$,  whereas any physical system is necessarily finite. In tridiagonal models \cite{HatanoNelson1996, ammari2024spectra}, that is, for $m = 1$, the GBZ is parametrised by a scaled unit circle, $p(e^{\i\alpha}) =re^{\i\alpha}$ and the transform $\mathbf{T}_n(f_1) \to \mathbf{T}_n(f_1\circ p)$ is an exact similarity for every $r>0$ and $n \in \N$. In effect, the GBZ, although defined through the limit $n\to \infty$, can be used interchangeably for finite systems. For systems with long-range coupling, however, this may no longer hold, as the GBZ is in general no longer just a scaled unit circle, but rather a homotopically deformed one. Banded non-Hermitian Toeplitz matrices arise in applications ranging from subwavelength resonance problems \cite{ammari2023nonhermitianskineffectthreedimensional, debruijn2025spectrapseudospectranonhermitiantoeplitz} to non-local Hamiltonians \cite{WegenerNonlocal, KadicNonlocal}.
The question is therefore in which sense, and at which rate, the finite system converges to this asymptotic description. Our results give the following answer. At a finite matrix level, the transform $\mathbf{T}_n(f_m) \to \mathbf{T}_n(f_m\circ p_n^{(K)})$ is no longer an exact similarity, but as we have seen in Section~\ref{ssec:higher}, the error is controlled and decays like $n^{-K}$ in the matrix size $n$ for any fixed order $K$.
This concept is schematically illustrated in Figure~\ref{fig: Spectral connection}. 

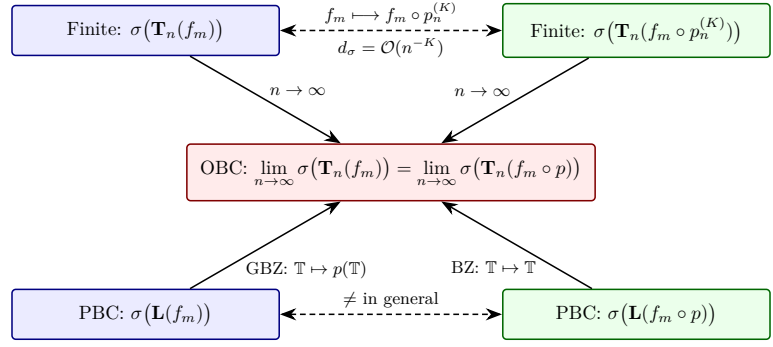
\begin{figure}[h]
    \centering
    \begin{adjustbox}{width=0.62\linewidth} 
    
\begin{tikzpicture}[
    box/.style={draw, align=center, inner sep=7pt, thick, rounded corners=2pt,
                minimum width=4.9cm},
    arr/.style={-{Stealth[length=3mm]}, thick}
]

\node[box, fill=blue!8, draw=blue!50!black] (finf) at (0, 5.0)
    {Finite: $\sigma\bigl(\mathbf{T}_n(f_m)\bigr)$};
\node[box, fill=green!8, draw=green!40!black] (fing) at (9.0, 5.0)
    {Finite: $\sigma\bigl(\mathbf{T}_n(f_m \circ p_n^{(K)})\bigr)$};
\node[box, fill=red!8, draw=red!50!black] (open) at (4.5, 2.4)
    {OBC: $\lim\limits_{n\to\infty}\sigma\bigl(\mathbf{T}_n(f_m)\bigr)
     = \lim\limits_{n\to\infty}\sigma\bigl(\mathbf{T}_n(f_m\circ p)\bigr)$};
\node[box, fill=blue!8, draw=blue!50!black] (perf) at (0, -0.2)
    {PBC: $\sigma\bigl(\mathbf{L}(f_m)\bigr)$};
\node[box, fill=green!8, draw=green!40!black] (perg) at (9.0, -0.2)
    {PBC: $\sigma\bigl(\mathbf{L}(f_m\circ p)\bigr)$};
\draw[arr] (finf) -- (open)
    node[midway, above right=-2pt, font=\small] {$n\to\infty$};
\draw[arr] (fing) -- (open)
    node[midway, above left=-2pt, font=\small] {$n\to\infty$};
\draw[arr] (perf) -- (open)
    node[pos=0.35, below right=-3pt, font=\small] {GBZ: $\mathbb{T}\mapsto p(\mathbb{T})$};
\draw[arr] (perg) -- (open)
    node[pos=0.35, below left=-3pt, font=\small] {BZ: $\mathbb{T}\mapsto\mathbb{T}$};
\draw[{Stealth[length=3mm]}-{Stealth[length=3mm]}, thick, densely dashed]
    (finf) -- (fing)
    node[midway, above, font=\small] {$f_m \longmapsto f_m\circ p_n^{(K)}$}
    node[midway, below, font=\small] {$d_\sigma = \mathcal{O}(n^{-K})$};

    \draw[{Stealth[length=3mm]}-{Stealth[length=3mm]}, thick, densely dashed]
    (perf) -- (perg)
    node[midway, above, font=\small] {$\neq$ in general};
\end{tikzpicture}

    \end{adjustbox}
    \caption{Relating OBC to PBC has been well established, which both describe infinite or semi-infinite systems. However everything an experiment or a computation touches is on the finite level which we relate through the following diagram.}
    \label{fig: Spectral connection}
\end{figure}

The generalised Brillouin zone does not deform as the chain is shortened, as the correction is only tangential. What changes is only the set of momenta at which it is sampled, and the leading displacement is tangential, of order $1/n$, and available in closed form as the boundary phase $\Delta\Psi$. 
Clearly, the parametrisations of the GBZ converge uniformly in the limit as $n\to \infty$, that is, $p_n\to p$.

\section*{Acknowledgments}
The authors thank Michael Floater for insightful discussions, which greatly improved the numerical methods of this work.

\section{Data availability} \label{Sec: Data availability}
The \texttt{Matlab} code for the numerical experiments developed in this work is openly available in the following repository:
\url{https://github.com/yannick2305/ToeplitzSymmetriser}.

\printbibliography

\appendix

\section{Effects of truncating the symmetrised symbol function}\label{Sec: symbol truncation}

We now show that $\mathbf{T}_n(g)$ may be replaced by the $M$-banded matrix $\mathbf{T}_n(g_M)$, where $g_M(z):=\sum_{|k|\le M}b_kz^{-k}$, at a cost that is exponentially small in $M$. Let us define
\begin{equation}\label{eq: eps_M}
    \varepsilon_M:=\sum_{|k|>M}|b_k|\le\frac{2Ce^{-\gamma_F M}}{e^{\gamma_F}-1},
\end{equation}
and order the eigenvalues of $\mathbf{T}_n(g)$ and $\mathbf{T}_n(g_M)$ increasingly, $\lambda_1\le\dots\le\lambda_n$ and $\tilde\lambda_1\le\dots\le\tilde\lambda_n$, with unit eigenvectors $\mathbf{v}_j$ and $\mathbf{\tilde v}_j$.

\begin{theorem}[Banded approximation]\label{thm: banded approximation}
There exist $c_g>0$ and $n_0\in\N$, depending only on $g$, such that for all $n\ge n_0$, all $M\in\N$ and all $1\le j\le n$:
\begin{enumerate}
    \item[(i)] $\bigl|\lambda_j-\tilde\lambda_j\bigr|\le\varepsilon_M$;
    \item[(ii)] if $\varepsilon_M\le c_g/(2(n+1)^2)$, then $\tilde\lambda_j$ is simple and, with the sign of $\tilde v_j$ chosen such that $\langle v_j,\tilde v_j\rangle\ge0$,
    \begin{equation}\label{eq: eigenvector bound}
        \|v_j-\tilde v_j\|_2\le\frac{2\sqrt{2}(n+1)^2}{c_g}\varepsilon_M .
    \end{equation}
\end{enumerate}
In particular, the eigenvalues of the $M$-banded matrix approximate those of $\mathbf{T}_n(g)$ up to $\mathcal{O}(e^{-\gamma_F M})$ uniformly in $n$, and the eigenvectors converge uniformly in $j$ as soon as $M-\frac{2}{\gamma_F}\log n\to\infty$.
\end{theorem}

\begin{proof}

Since $g$ is real-valued on $\mathbb{T}$ and $p$ is symmetric under $\theta\mapsto-\theta$, the coefficients are real and $b_{-k}=b_k$, so that $\mathbf{T}_n(g)$ is real symmetric. Writing $g(\theta):=g(e^{\i\theta})$, Assumption~\ref{ass:standing} together with Lemma~\ref{lem: separation implies non confluence} implies that $g$ is simple-loop. It is strictly monotone on $[0,\pi]$ with $g'\neq0$ on $(0,\pi)$ and $g''(0)\neq0\neq g''(\pi)$.

We start by proving the eigenvalue estimate. The difference $\mathbf{E}:=\mathbf{T}_n(g)-\mathbf{T}_n(g_M)=\mathbf{T}_n(g-g_M)$ is real symmetric. For any $h\in L^\infty(\mathbb{T})$ one has $\|\mathbf{T}_n(h)\|_2\le\|h\|_{L^\infty(\mathbb{T})}$, since $\mathbf{T}_n(h)=\mathbf{P}_n\mathbf{L}(h)\mathbf{P}_n$ is a compression of the Laurent operator $\mathbf{L}(h)$ on $\ell^2(\Z)$, which is unitarily equivalent to multiplication by $h$ on $L^2(\mathbb{T})$. Hence
\begin{equation}\label{eq: E bound}
    \|\mathbf{E}\|_2\le\|g-g_M\|_{L^\infty(\mathbb{T})}\le\sum_{|k|>M}|b_k|=\varepsilon_M .
\end{equation}
The pointwise eigenvalue bound established in (i) is therefore a direct consequence of Weyl's inequality for symmetric matrices.

To show the eigenvector estimate given in (ii) we proceed as follows. Following \cite{BOGOYA_SLM} the eigenvalues are sampled from the symbol function, so

\begin{equation}
    \lambda_{j+1} - \lambda_j = C_1 \frac{g'(\theta)}{n} + C_2 \frac{g''(\theta)}{n^2} + \mathcal{O}(n^{-3}) \geq Cn^{-2}.
\end{equation}
Let us define the minimal spectral distance as follows,
\begin{equation}\label{eq: gap}
    \delta_n(g):=\min_{i\neq j}|\lambda_i-\lambda_j|\geq c_g(n+1)^{-2}.
\end{equation}

By assumption of the theorem $\varepsilon_M\le c_g/(2(n+1)^2)\le\delta_n(g)/2$ and by the first part of the proof 
\begin{equation}
    |\tilde\lambda_i-\lambda_j|\ge|\lambda_i-\lambda_j|-\varepsilon_M\ge\delta_n(g)/2, \quad i \neq j.
\end{equation}
In particular the $\tilde\lambda_i$ are pairwise distinct, so each $\tilde v_j$ is determined up to sign. Expand $v_j=\sum_{i}c_i\tilde v_i$ with $\sum_i c_i^2=1$, choosing the sign of $\tilde v_j$ such that $c_j=\langle v_j,\tilde v_j\rangle\ge0$. Using $\mathbf{T}_n(g)v_j=\lambda_jv_j$ and \eqref{eq: E bound},
\begin{equation}
    \varepsilon_M^2\ge\|\mathbf{E}v_j\|_2^2=\|(\mathbf{T}_n(g_M)-\lambda_j)v_j\|_2^2
    =\sum_{i}(\tilde\lambda_i-\lambda_j)^2c_i^2
    \ge\frac{\delta_n(g)^2}{4}\sum_{i\neq j}c_i^2
    =\frac{\delta_n(g)^2}{4}(1-c_j^2).
\end{equation}
Since $c_j\ge0$, the difference of the two unit vectors is controlled by
\begin{equation}
    \|v_j-\tilde v_j\|_2^2=2(1-c_j)\le2(1-c_j)(1+c_j)=2(1-c_j^2).
\end{equation}
Combining the last two displays with \eqref{eq: gap} yields
\begin{equation}
    \|v_j-\tilde v_j\|_2^2\le\frac{8\varepsilon_M^2}{\delta_n(g)^2}\le\frac{8(n+1)^4}{c_g^2}\varepsilon_M^2,
\end{equation}
and \eqref{eq: eigenvector bound} follows by taking square roots. The final assertion follows by inserting \eqref{eq: eps_M} into \eqref{eq: eigenvector bound}.
\end{proof}

Consequently, all eigenvalue statements of Section~\ref{sec:sharp} remain valid for the truncated symbol $g_M$ up to the additive error $\varepsilon_M$ of Theorem~\ref{thm: banded approximation} (i) uniformly in the matrix size $n$.
The error in the eigenvectors introduced by the truncation is controlled by Theorem~\ref{thm: banded approximation} (ii) and grows quadratically in $n$.

\end{document}